\documentclass[12pt]{amsart}
\usepackage{latexsym,fancyhdr,amssymb,color,amsmath,amsthm,graphicx,listings,comment}
\usepackage[section]{placeins}
\newtheorem{thm}{Theorem} 
\newtheorem{lemma}{Lemma} 
\newtheorem{propo}{Proposition} 
 \newtheorem{coro}{Corollary}
\definecolor{red1}{rgb}{1,0.9,0.9} \definecolor{blue1}{rgb}{0.9,0.9,1} \definecolor{green1}{rgb}{0.9,1,0.9} 
\definecolor{yellow1}{rgb}{1,1,0.9} \definecolor{yellow2}{rgb}{1,1,0.8}

\let\paragraph\subsection
\newcommand{\NN}{\mathbb{N}}
\newcommand{\RR}{\mathbb{R}}
\newcommand{\CC}{\mathbb{C}}
\newcommand{\Zeta}{\mathrm{Z}}

\title{On Helmholtz free energy for finite abstract simplicial complexes}
\author{Oliver Knill}
\date{Mar 18, 2017}
\address{Department of Mathematics \\ Harvard University \\ Cambridge, MA, 02138 }

\begin{document}
\maketitle

\begin{abstract}
We prove first that for the Barycentric refinement $G_1$ 
of a finite abstract simplicial complex $G$, the Gauss-Bonnet formula 
$\chi(G) = \sum_x K^+(x)$ holds, where $K^+(x)=(-1)^{\dim(x)} (1-\chi(S(x)))$
is the curvature of a vertex $x$ with unit sphere $S(x)$ in the graph $G_1$. 
This curvature is dual to $K^-(x)=(-1)^{\dim(x)}$ for which Gauss-Bonnet is the definition of
Euler characteristic $\chi(G)$. 
Because the connection Laplacian $L'=1+A'$ of an abstract simplicial complex $G$ is unimodular, 
where $A'$ is the adjacency matrix of the connection graph $G'$, 
the Green function values $g(x,y) = (1+A')^{-1}_{xy}$ are integers and $1-\chi(S(x))=g(x,x)$.
Gauss-Bonnet for $K^+$ reads therefore as ${\rm str}(g)=\chi(G)$, where ${\rm str}$ is the super trace. 
As $g$ is a time-discrete heat kernel, this is a cousin to McKean-Singer ${\rm str}(e^{-Lt}) = \chi(G)$ 
for the Hodge Laplacian $L=(d+d^*)^2$ which lives on the same Hilbert space than $L'$. 
Both formulas hold for an arbitrary finite abstract simplicial complex $G$.
Writing $V_x(y)= g(x,y)$ for the Newtonian potential of the connection Laplacian,
we prove $\sum_y V_x(y) = K^+(x)$, so that by the new Gauss-Bonnet formula, 
the Euler characteristic of $G$ agrees with the total potential theoretic
energy $\sum_{x,y} g(x,y)=\chi(G)$ of $G$. The curvature $K^+$ now relates to the probability 
measure $p$ minimizing the internal energy $U(p)=\sum_{x,y} g(x,y) p(x) p(y)$ of the complex. 
Since both the internal energy (here linked to topology) and Shannon entropy are natural
and unique in classes of functionals, we then look at critical points $p$ the 
Helmholtz free energy $F(p)=\beta U(p)-T S(p)$ which combines
the energy functional $U$ and the entropy functional $S(p)=-\sum_x p(x) \log(p(x))$. 
As the temperature $T=1-\beta$ changes, we observe bifurcation phenomena. 
Already for $G=K_3$ both a saddle node bifurcation and a pitchfork bifurcation occurs.
The saddle node bifurcation leads to a catastrophe:
the function $\beta \to F(p(\beta),\beta)$ is discontinuous if $p(\beta)$ is a free energy minimizer.
\end{abstract}

\section{Introduction}

\paragraph{}
To every geometry with a Laplacian belongs a {\bf Newtonian potential theory}. The
prototype was developed by Gauss for the Laplacian $(4\pi)^{-1} \Delta$ in $\RR^3$, where the Green
function $g(x,y)=V_x(y) = -1/|x-y|$ defines the familiar Newton potential making its appearance 
in classical gravity and electro statics. 
As calculus shows,
the Gauss law ${\rm div}(F) = d^* d V = {\rm div}({\rm grad}(V)) = \Delta V= \mu$ 
determines the gravitational potential $V$ of a mass distribution $\mu$ and the gravitational 
force $F = {\rm grad}(V)$. An other important classical case with applications in vortex dynamics
or complex dynamics or spectral theory of self adjoint operators is the {\bf logarithmic potential} 
$g(x,y)=V_x(y) = \log|x-y|$ associated to the Laplacian $(2\pi)^{-1} \Delta$ in $\RR^2$. 
Given a measure $\mu$, it defines the energy $I(\mu)= \int_\CC \int_\CC \log|x-y| \mu(x) d\mu(y)$. 
The logarithmic capacity of $K \subset \CC$ is then the minimum of $e^{-I(\mu)}$ over all 
probability measures $\mu$ supported on $K$. 
For discrete measures $\mu$, one usually disregards the self interaction.
One can then minimize $I(\mu) = \int_{x \neq y} \log|x-y| \mu(x) d\mu(y)
= \log \prod_{\lambda_j \neq \lambda_k} |\lambda_j-\lambda_k|$, when $\mu$ ranges over all discrete 
probability measures. In a spectral or random matrix setting, this energy appears as a van 
der Monde determinant. In one dimensions, the potential energy is in statistics 
known as the {\bf Gini index} $I(\mu) = \int_{\RR} \int_{\RR} |x-y| d\mu(x) d\mu(y)$ 
because $V_x(y)=|x-y|$ is the natural Newton potential to the Laplacian 
$(1/2) \Delta = (1/2) d^2/dx^2$ on the real line. 

\paragraph{}
A finite simplicial complex $G$ carries an exterior derivative $d$ given as an 
incidence matrix. It defines the {\bf Hodge Laplacian} $L=(d+d^*)^2$ and has so a Newtonian potential 
theory. But as in the manifold case, we have to deal with singularities, as this Laplacians is not
invertible. Indeed, both for compact manifolds as well as for finite complexes, 
the kernels of of the blocks $L_k$ of $L$ consist of harmonic forms which by Hodge 
are as vector spaces isomorphic to the $k$'th cohomology $H^k(G)$ of the simplicial complex $G$. 
While this is topologically interesting, the regularization 
via a pseudo inverse renders the individual entries less likely to be of
topological interest. By the spectral theorem for selfadjoint matrices, using an 
orthonormal eigenbasis $\psi_k$ the psudeo inverse can be written as
$$ g(x,y) = \sum_{\lambda_k \neq 0} \frac{\psi_k(x) \psi_k(y)}{\lambda_k} \; . $$
And even if regularized by restricting $L$ to the orthogonal complement of the kernel, 
the Green functions $g(x,y)$ of successive Barycentric refinements
$G_n$ explode as there is no spectral gap at $0$ in the limit $G_{\infty}$. 

\paragraph{}
While the dynamical importance of the inverse Hodge Laplacian $L$ is evident, the Green function
values are just real numbers, and even if existent, a topological connection would be 
difficult to detect. We have in vain tried to associate the {\bf Hodge energy} $\sum_{x,y} L^{-1}(x,y)$ 
defined by the Hodge Laplacian $L$ with anything topological of $G$. 
The situation completely changes for the {\bf connection Laplacian} $L'=1+A'$, for which
the Green function entries $g(x,y) = (1+A')^{-1}_{xy}$ are integers. 
The diagonal entries are Poincar\'e-Hopf type indices and can also be seen as a generalized 
genus of a unit sphere $S(x)$ and related to curvatures $K^+(x) (-1)^{{\rm dim}(x)}$. 
The set of these values does not change any more under further Barycentric refinements 
and are therefore combinatorial invariants. 

\paragraph{}
This note is 
a continuation of \cite{Unimodularity,Spheregeometry} and was obtained by studying the topological 
nature of $g(x,y)$ for $x \neq y$. 
There is still an enigma about these off-diagonal entries.
Since the diagonal entries $g(x,x)$ are topological and given 
by $g(x,x) = 1-\chi(S(x))$, where $S(x)$ is the unit sphere of $x$ in $G_1$, we expect also 
all other entries to have topological interpretation. There are some indication that this
is so as we see that the intersection of unstable manifolds $W^+(x)$ and $W^+(x)$ for the gradient 
flow of the dimension functional on $G_1$ needs to be non-empty. As we will see in the proof of our
main result $\sum_{x,y} g(x,y) = \chi(G)$, the entries $g(x,y)$ are closely related to 
Poincar\'e-Hopf indices which appear here as curvatures. It was the discovery of this 
identity ``energy = Euler characteristic" which led us to look into the 
thermodynamic branch of the story. 

\paragraph{}
Diving into the potential theoretical aspect we establish new relations for the Green function values. 
If the potential energy $\sum_{x,y} g(x,y)$ is seriously considered to be 
an internal energy in the sense of physics (like the energy of a finite vortex configuration in the case of 
the Laplacian in $\RR^2$ or the energy of a finite electron configuration in the case of the Laplacian in $\RR^3$),
this naturally leads to thermodynamics and in particular to questions usually studied in 
statistical mechanics. When considering arbitrary simplicial complexes $G$, it would 
be natural to look at the Barycentric refinement limits $G_{\infty}$ and hope that the limiting 
case leads to situations which are universal at critical parameters. On a spectral level there is 
some universality already in the sense that the law of the complex, the density of states converges
universally to a limit which only depends on dimension \cite{KnillBarycentric,KnillBarycentric2}
The Barycentric limit replaces the van-Hove limit in lattice gas models 
\cite{RuelleStatMech,SimonStatMechanics}. But as we will see, already the bifurcation story
for the Helmholtz free energy functional for a fixed network $G$ can be complicated  - 
without any thermodynamic limit.
For the Whitney complex $G$ of the graph $K_3$, the simplest two-dimensional complex, the
 operators $L,L'$ are $7 \times 7$ matrices and two important bifurcation cases known in 
dynamical systems theory appear. 

\section{Gauss Bonnet for $G_1$} 

\paragraph{}
Let $G$ be an abstract finite simplicial complex. This means that $G$ a collection of finite, non-empty sets closed
under the operation of taking non-empty subsets. It defines two finite simple graphs $G_1$ and $G'$ which 
both have the faces of $G$ as vertex set. In the Barycentric refinement $G_1$, two vertices 
are connected if one is contained
in the other, in $G'$ two vertices are connected if they intersect. If $A'$ is the adjacency matrix of $G'$
then the Fredholm connection Laplacian $L'=1+A'$ is unimodular so that the Green functions $g(x,y) = L'^{-1}_{xy}$
are integers. We know that $g(x,x) = 1-\chi(S(x))$, where $S(x)$ is the unit sphere of $x$ in $G_1$. 
The unit sphere $S(x)$ in $G_1$ the graph generated by the vertices in $G_1$ which have distance $1$ to $x$. 

\paragraph{}
As usual in potential theory, $V_x(y)=g(x,y)$
is the {\bf Newtonian potential} at $y$ if a unit mass is placed at $x$. If we place a unit mass at every vertex, then
$K^+(x)=\sum_y g(x,y)$ is the {\bf potential energy} at $x$. We call this the {\bf unstable curvature} at $x$. The stable
curvature is defined as $K^-(x)= (-1)^{\dim(x)}$ and relatively plain. But they are dual to each other as one belongs to 
$f(x)={\rm dim}(x)$ and the other to $f(x) =-{\rm dim}(x)$. The two curvatures agree if $G$ is discrete even dimensional 
manifold. Also, due to index averaging results, we will see that if $G$ was the Whitney complex of a graph $(V,E)$, 
then pushing either of them to the vertices gives the Euler curvature. 

\begin{thm}
For any finite abstract simplicial complex $G$, the Gauss-Bonnet formula 
$$\sum_x K^+(x) = \chi(G) \;  $$ 
holds, where the sum is taken over the faces of $G$. 
\end{thm}
\begin{proof}
The unit sphere $S(x)$ in $G_1$ is the Zykov join $S^-(x) + S^+(x)$ of the stable and 
unstable sphere defined as $S^-(x) = \{ y \in S(x) \; | \; x \subset y \}$
and $S^+(x) = \{ y \in S(x) \; | \; y \subset x \}$. We know that
the functional $i(G) = 1-\chi(G)$ is multiplicative for the join $i(S(x)) = i(S^+(x)) i(S^-(x))$.
But because $S^-(x)$ is the $({\rm dim}(x)-1)$-skeleton of the simplex $x$ and so a $({\rm dim}(x)-1)$-sphere
which has $\chi(S^-(x)) = 1 - (-1)^{1+{\rm dim}(x)}$ so that $i(S^-(x))=(-1)^{{\rm dim}(x)}$
and the Poincar\'e-Hopf index $i^+(x) = 1-\chi(S^+(x))$ is the curvature 
$(1-\chi(S(x))) (-1)^{\rm dim}(x)$. Now use Poincar\'e-Hopf. 
\end{proof}

\paragraph{}
As the punchline of the proof was Poincar\'e-Hopf, this Gauss-Bonnet result is just a 
Poincar\'e-Hopf result \cite{poincarehopf} in disguise. It 
belongs to the Morse function $f(x)=-{\rm dim}(x)$ which is locally injective (a coloring) 
on $G_1$. Its structure is in general more
interesting than the stable curvature $K^-(x)$, which is like $K^+(x)$ a divisor on $G$.
The Gauss-Bonnet result for $K^-(x)=(-1)^{{\rm dim}(x)}$ is
essentially the definition of Euler characteristic of $G$ and known to coincide with $\chi(G_1)$.  
Note that $\chi(G')$ is in general different from $\chi(G)$. For the Whitney complex $G$ 
of the octahedron for example, $G'$ has Euler characteristic $0$ and is actually homotopic to a 
3-sphere. 

\paragraph{}
We have seen that averaging Poincar\'e-Hopf over natural probability spaces gives
the Euler curvature \cite{indexexpectation,knillgraphcoloring}. 
One of the simplest averages is $j_f(x)=(i_f^+(x)+i_f^-(x))/2$ which is of topological interest.
For odd-dimensional $x$, this is $0-\chi(S(x))/2$. 
For even dimensional $x$, this is $1-\chi(S(x))/2$. In \cite{indexformula} we have seen that 
$j_f(x)=1-\chi(S(x))/2 - \chi(B_f(x))/2 (*)$,
where $B_f(x)=\{ y \; | \; f(y)=f(x) \}$ is a discrete contour-surface in the sense 
of \cite{KnillSard}. Therefore, for $f={\rm dim}$ and even dimensional $x$, where 
$i_f^+(x)=1-\chi(S(x)), i_f^-(x)=1$ we have $\chi(B_f(x))=0$.
If $x$ is odd dimensional, where $i_f^+(x)=\chi(S(x))-1$, $i_f^-(x)=-1$, then 
$j_f(x) = \chi(S(x))/2-1$ and so $\chi(B_f(x)) = \chi(S(x))/2-1$. 
If $G$ is a $d$-graph, where all unit spheres $S(x)$ are $(d-1)$-spheres,
then $B_f(x)$ are $d-2$ spheres, then if $x$ is even dimensional $j_f(x)=1-\chi(S(x))/2=1$
and if $x$ is odd dimensional, $j_f(x)=\chi(S(x))/2-1=-1$. But the point of the 
index formula (*) was that in the four dimensional case and any $f$, the index can be written
in terms of $1-\chi(B_f(x))/2$, where $B_f(x)$ is a $2$-dimensional graph so that by applying
Gauss-Bonnet and index-averaging the Euler characteristic is related to an average sectional
curvature. 

\section{McKean-Singer} 

\paragraph{}
The {\bf super trace} of a matrix $A$ acting on $\RR^n$, where $n$ is the
number of faces of $G$ is defined as the super sum of $A$ over the diagonal: 
$$   {\rm str}(A) = \sum_x (-1)^{{\rm dim}(x)} A_{xx} \; . $$
For example, the super trace of the identity operator is by definition 
the Euler characteristic 
$$ {\rm str}(1) = \sum_x (-1)^{{\rm dim}(x)} = \chi(G)  \; . $$
And since the connection Laplacian has $1$ in the diagonal, we know this also for 
the connection Laplacian $L'=1+A'$ 
$$ {\rm str}(L') = \chi(G) \; . $$

Not so obvious is the following McKean-Singer interpretation for the Green functions
$g=L'^{-1}$:

\begin{coro}
If $g$ is the Green function of a simplicial complex $G$, then ${\rm str}(g)= \chi(G)$
\end{coro}
\begin{proof}
The diagonal elements of $g_{xx}=(1+A')^{-1}_{xx}$ are 
$1-\chi(S(x))$ \cite{Spheregeometry}. The Gauss-Bonnet theorem shows that 
$$  {\rm str}(g) = \sum_x (-1)^{{\rm dim}(x)} g_{xx}  = \chi(G) \; . $$
\end{proof}

\paragraph{}
Since the {\bf discrete time heat equation} $u(n+1)-u(n) = A' u(n)$
means $u(n-1)=(1+A')^{-1} u(n)$, this formula is a cousin of the
usual McKean-Singer formula ${\rm str}(\exp(-L)) = \chi(G)$
for the Hodge-Laplacian $L=(d+d^*)^2$ of $G$. This result for manifolds
\cite{McKeanSinger} was ported to finite geometries in \cite{knillmckeansinger}
by adapting the super symmetry proof \cite{Cycon}. 
The matrices $A'$ and $L$ have the same size and both 
$$  \exp(-L)=1-L+L^2/2!-\dots$$ 
and 
$$  (1+A')^{-1} = 1-A+A^2-\dots $$
have the same super trace $\chi(G)$. The second geometric sum is divergent
and has to be understood as an analytic continuation of the matrix-valued function
$$ \Zeta(z) = (1+z A')^{-1} $$
which is analytic for complex $|z|<1$ and which has as a determinant
the Bowen-Lanford zeta function \cite{BowenLanford} 
$$ \zeta(z) = \det(\Zeta(z)) \; . $$

\begin{figure}[!htpb]
\scalebox{0.25}{\includegraphics{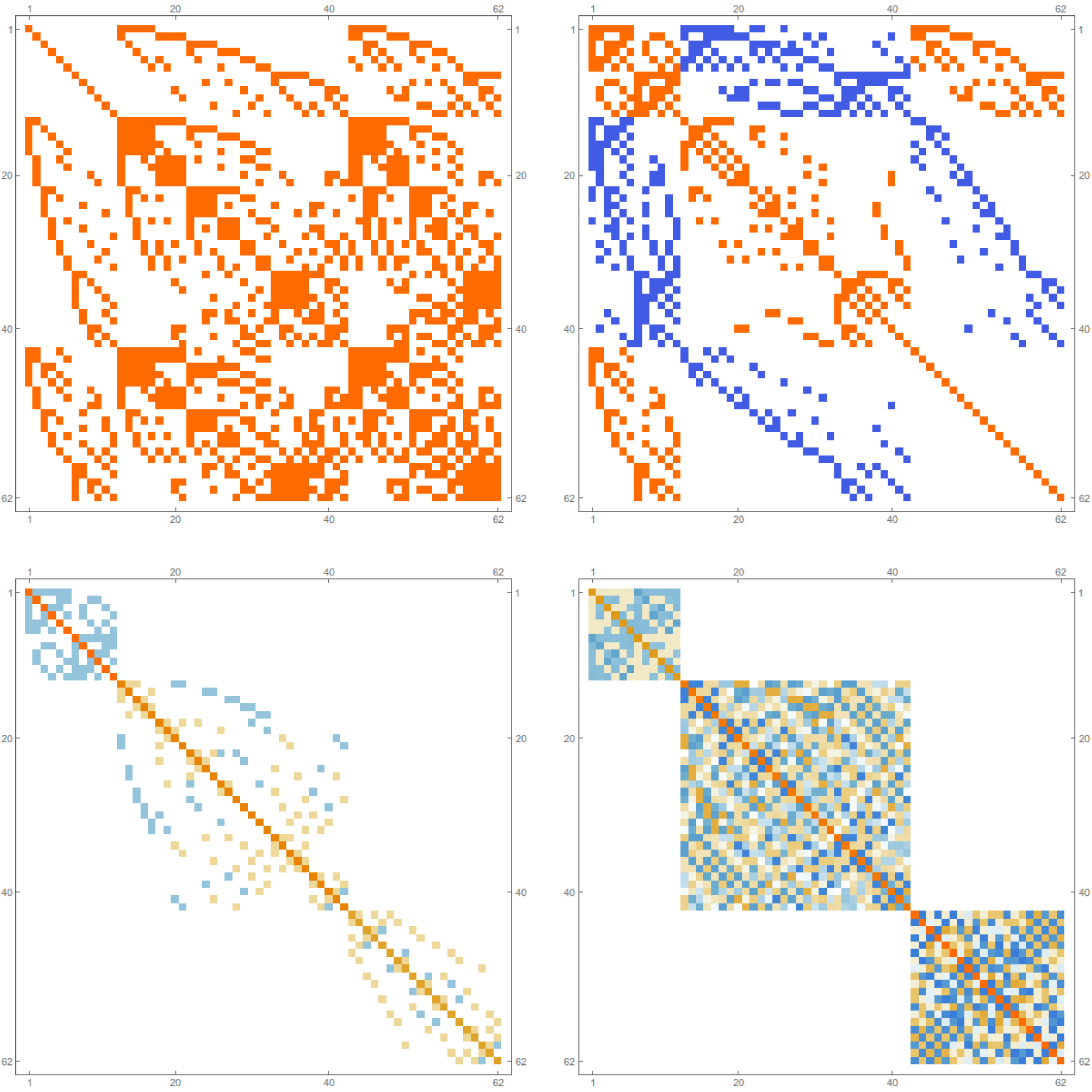}}
\caption{
Above we see the Fredholm connection matrix $L'=(1+A')$ and its inverse $g=(L')^{-1}$
in the case of the icosahedron graph $G$. 
Below is the Hodge Laplacian $L=(d+d^*)^2$ and its pseudo inverse $L^{-1}$. 
All matrices are $62 \times 62$ matrices as the f-vector of $G$ is 
$(v,e,f)=(12,30,20)$ so that $G$ has $12+30+20=62$ faces. The Hodge
Laplacian $L=(d+d^*)$ is block diagonal with
Kirchhoff Laplacian $L_0=d_0^* d_0$, the operator $L_1=d_1^* d_1 + d_0 d_0^*$
and $L_2=d_2 d_2^*$. The diagonal elements $g(x,x)=1-\chi(S(x))=1$
are topological. The total energy $\sum_{x,y} g(x,y) = 2$ is $\chi(G)=12-30+20$. 
The diagonal elements of the pseudo inverse $L^{-1}$ has 
diagonal elements $7/36$ in $L_0^{-1}$ and $86/225$ on $L_1^{-1}$ and $L_2^{-1}$.
Its super trace of $L^n$ is zero for all $n \neq 0$ by McKean-Singer super symmetry:
the union of the non-zero spectra $\sigma(L_0)=\{ (5 \pm \sqrt{5})^{(3)},6^{(5)} \}$
and $\sigma(L_2) = \{ (3\pm \sqrt{5})^{(3)},5^{(4)},3^{(4)},2^{(5)}\}$ coincides with 
the non-zero spectrum $\sigma(L_1)$. }
\end{figure}

\section{Newtonian potential}

\paragraph{}
The {\bf Newtonian potential} of a simplex $x \in V(G_1)$ is defined as the function 
$V_x(y) = g(x,y)$. For a measure $p$ onto the vertex set of $G_1$, the 
{\bf potential of the measure $p$ at $x$} is 
$$ V_x(p) = \sum_y p(y) V_x(y) \; . $$
The {\bf total energy} of the probability measure $p$ on $G$ is 
$$ H(p)  = \sum_{x} V_x(p) = \sum_{x,y} V_x(y) p(x) p(y)  = \sum_{x,y} g(x,y) p(x) p(y) \; . $$
We want now to show that if we put equal mass $1$ to each of the
points $x$, then $V_x(1) = \sum_y V_x(y)$ agrees with unstable curvature. 
In order to prove this we first give a formula for the 
{\bf connection vertex degree} $d_{G'}(x)$ of $x$ in terms of stable spheres 
$S^+(y) = \{ z \in S(y)\; | \; x \subset y \}$ in $G_1$. We denote
by $B_{G'}(x)$ the {\bf unit ball} of $x$ in the connection graph $G'$. 

\begin{lemma}
$d_{G'}(x)= \sum_{y \in B_{G'}(x)} \chi(S^+(y))$. 
\label{balllemma}
\end{lemma}
\begin{proof}
Since $\chi(S^+(y)) = (-1)^{{\rm dim}(y)} \chi(S(y))$, this is equivalent to 
$$ d(x) = \sum_{y \in B(x)} (1+(-1)^{{\rm dim}(y)}) + (-1)^{{\rm dim}(y)} \chi(S(y))  \; . $$
This can now be rewritten as 
$$  \chi(S(x)) = \sum_{y \in S(x)} (-1)^{{\rm dim}(y)} \chi(S(y)) $$
which is Gauss-Bonnet for the unit sphere $S_{G'}(x)$.
\end{proof}

\paragraph{}
Lets look at some examples. \\
1) If $G=K_3$ and $x$ is the central vertex of $G_1$ of dimension $2$, 
then $d(x)=6$ and every $\chi(S^+(y))$ in $B(x)$ except $x$ itself has $\chi(S^+(y))=1$. 
2) In a general complex $G$, If $x$ is a facet in $G$ (a face of maximal dimension), 
then $\chi(S^+(x))=0, \chi(S^+(y))=1$ for all neighbors. 

\paragraph{}
Here is the curvature interpretation of the row sum or column sum of the Green function 
matrix $g$. This was first found experimentally and remained a mystery for some time
until realizing that one does not have to know explicit formulas for $g(x,y)$ but only needs
to know about $g(x,x)$: 

\begin{propo}
$\sum_y g(x,y) = (-1)^{{\rm dim}(x)} g(x,x)$. 
\end{propo}
\begin{proof}
Since the right hand side is $K^+(x)$, this 
this can be restated in vector form as $g 1 = K^+$, where $1$ is the vector
which is constant $1$. As 
$g=(1+A')^{-1}$, this is equivalent to $(1+A') K^+ = 1$.
But since $K^+(x)=1-\chi(S^+(y))$ and $A'$
is the adjacency matrix of $G'$, this means
$$ (1+A') (1-\chi(S^+(y))) = 1 + d_{G'}(x) - \sum_{y \in S(x)} \chi(S^+(y)) = 1 \; . $$
Now use Lemma~(\ref{balllemma}).
\end{proof}

\paragraph{}
This proposition implies that if $x$ has even dimension, then the potential contribution
from points outside $x$ is zero and that for odd-dimensional $x$, the 
potential energy contributions from points outside of $x$ is equal to the internal
energy $g(x,x)$ of the point $x$. 

\paragraph{}
We immediately get a formula for the total energy. Also this relation was first found 
experimentally and triggered writing down this note: 

\begin{coro}
$\sum_{x,y} g(x,y) = \chi(G)$. 
\end{coro}
\begin{proof}
Use the lemma and the Gauss-Bonnet result.
\end{proof}

\paragraph{}
In some sense, the measure of maximal entropy leads to an energy which is 
topological. When rescaled by the size of the network, we get a number which 
is independent of Barycentric refinements. 
We have seen in other places \cite{indexformula,KnillFunctional} 
that the Euler characteristic has
some affinities with the Hilbert action, the sum over all sectional curvatures
at a point. Seeing both classical and relativistic connections between 
gravity and Euler characteristic makes the functional even more interesting. 

\section{Minimizing energy} 

\paragraph{}
The problem of minimizing potential energy
$$ F(p) = (g p,p) = \sum_{x,y} p(x) p(y) g(x,y)   $$
among probability measures $p$ is a Lagrange problem as we
have the constraint $(p,1)=1$. 

\begin{lemma}
The minimal energy configurations have $p_k = (1+d(x))/Z$,
where $Z$ is the normalisation factor so that $\sum_k p_k =1$. 
\end{lemma}
\begin{proof}
The Lagrange equations are
$$ 2 g p = \lambda 1, (p,1) = 1 \; . $$
Since we can invert $g$, this gives $p = (1+A') 1 (\lambda/2)$
which means $p_k = (1+d(x)) (\lambda/2)$ 
The probability measure condition fixes $\lambda$. 
\end{proof}

\paragraph{}
More in the spirit of quantum mechanics is to write the probability
as the square of a {\bf wave function amplitude} $p=|\psi|^2$. We
minimize then $F(\psi)=(g \psi,\psi)$ under the constraint $(\psi,\psi)=1$.
The Lagrange problem is then a more familiar {\bf eigenvalue problem} 
$$ 2g \psi = \lambda \psi $$
and the Lagrange multiplier is an eigenvalue. In other words, $\psi$ is
then the {\bf Perron-Frobenius type eigenvector} and $\lambda$ is the 
largest eigenvalue of $g$ which means the lowest eigenvalue = {\bf ground state}
of the connection Laplacian $L'$. The magic is that this lowest eigenvalue 
is never $0$ and that this holds for an arbitrary abstract simplicial complex $G$.

\paragraph{}
More in the spirit of discrete Markov process is to look at the 
{\bf San Diego type Laplacian}  \cite{Chung97}
$\tilde{L}=P L' P$, where $P={\rm Diag}(1/\sqrt{p})$. Now $\psi=\sqrt{p}$  is an
eigenvector of the operator $\tilde{g}=\tilde{L}^{-1}$ with maximal eigenvalue
$c=1/\lambda_1$, where $\lambda_1$ is the {\bf minimal eigenvalue} of $L'$. 
This {\bf ground state energy} is also 
$$  \lambda_1 = \sum_{x,y} g(x,y) p(x) p(y) \; . $$
Compare this with 
$$  \chi(G)   = \sum_{x,y} g(x,y) \; . $$ 
and 
$$  \chi(G)   = \sum_{x,y} \tilde{g}(x,y) (1/\psi(x) ) (1/\psi(y))  \; . $$

\paragraph{}
Again we have to point out that minimizing energy classically for the potential
$-1/|x|$ of the Laplacian $-(4 \pi)^{-1} \Delta$ in $\RR^3$ does not make much sense due to
the unboundedness of the Green function. This also applies to the Green function $g$
the Hodge Laplacian $L$ for a finite abstract simplicial complex, where 
$(g \psi, \psi)$ can get arbitrary large
if $\psi$ gets close to a constant field, the minimizer of entropy.

\paragraph{}
What is unique about the connection 
Laplacian is that it leads to a {\bf natural quantization}, without the need for any 
regularization. The need for regularization penetrates classical field theories. 
What happens in the connection Laplacian is that the unimodularity theorem has shown
that independent of the network, the $0$ energy is uniformly off limits.
There is no need for numerical tricks nor the need for any regularizations. We would
not be surprised to see it to emerge to be relevant in some physics. 

\section{Shannon Entropy and Helmholtz free energy}

\paragraph{}
The {\bf Shannon entropy} of a probability measure $p$ on the set of faces
$x$ in a finite simplicial complex $G$ is defined \cite{Shannon48} as 
$$ S(p) = - \sum_x p(x) \log(p(x))  \; , $$
where the sum is over all faces $x$ of $G$.
The usual understanding is that if $p(x)=0$, then $p(x) \log(p(x))=0$. 
We will see that this case appears, especially 
at the zero temperature limit $\beta=1$ of the Helmholtz free energy
functional $F(p,\beta)$ of $G$. Actually, without defining any process of changing
$G$, the selection done by $p$ could get interesting geometries: 
run the variational problem to the Barycentric limit and hope for the appearance of some
universal $\beta$ value which leads to a space $G_{\infty}$ selected out by 
the limiting matter distribution $p_{\infty}$. Then study this geometry $G$. 

\paragraph{}
The Lagrange equations for the functional $S(p)$ under the constraint 
$(p,1)=1$ leads to the uniform distribution $p(x)=1/n$, where $n$ is the
number of simplices. With an adjusted inner product normalized so that then 
$(p,1)=1/n$, we can say:

\begin{coro}
The energy of the entropy maximizing measure is the Euler characteristic. 
\end{coro}

\paragraph{}
For a real non-negative temperature $T$, the {\bf Helmholtz free energy} is defined as 
$$ F(p,T) = H(p) - T S(p) \; , $$
where $H(p) = \sum_{x,y} g(x,y) p(x) p(y)$ is the {\bf internal energy}
and $S(p)$ is the Shannon entropy. The Helmholtz free energy allows classically
to compute all important thermodynamic properties of a system.
The equivalent functional $F(p,\beta) = \beta H(p) - S(p)$
is better suited for describing the {\bf infinite temperature} limit 
$\beta=0$, which means maximizing entropy. In order to study both limits, we will
always look at 
$$ F(p,\beta) = \beta H(p) - (1-\beta) S(p) \;  $$
so that with $\beta \in [0,1]$, we start with $\beta=0$ 
the {\bf infinite temperature limit} and end up with $\beta=1$, the 
{\bf zero temperature limit}. 

\paragraph{}
Is it possible that for the infinite temperature $\beta=0$ and zero temperature $\beta=1$ 
the critical points are the same? It would require $G'$ to be regular. 
But if $G$ is has maximal dimension larger than $0$, then the
connection graph $G'$ is never regular and the two measures differ. In the one-point
graph $K_1$, the function $F$ is constant. 

\paragraph{}
When using the original $F(p,T) = H(p) - T S(p)$, then the energy increases with $T$:

\begin{lemma}
For any simplicial complex $G$ and every probability measure $p$, 
the function $T \to F(p,T=H(p) - T S(p)$ is strictly monotone. 
\end{lemma}
\begin{proof}
We have $S(p)>0$ for any probability measure $p$
and it is maximized by $\log(n)$ if $G'$ has $n$ vertices.
The slope of the function $T \to F(p,T)$ is therefore in $[0,\log(n)]$. 
\end{proof} 

When fixing $p$ in $F(p,\beta)$ we have the partial derivative 
$F_{\beta}(p,\beta) = H(p) + S(p)$. But away from bifurcation values,
we can look at $\beta \to F(p(\beta),\beta)$. We see  
experimentally that $F_{\beta}(p(\beta),\beta) \leq 0$ and $F(p(\beta),\beta)>0$ at 
critical points but we have only looked at small cases so far. As we have seen earlier, the entries
$g(x,y)$ have both signs can have absolute values larger than $1$ for larger networks. 
Anyway, we would not bet yet on the observed $F_{\beta}(p,\beta) \leq 0$ and $F(p,\beta)>0$ but
just ask whether it is true. 

\paragraph{}
To the zero temperature energy limit $\beta=1$, we did not expect such a complicated behavior
at first because the Lagrange extremization problem gives a unique critical point, the Gibbs
distribution which is completely determined by the vertex degrees of $G_1$. 
Numerical explorations however show also other critical 
points in a limiting sense. What happens is that some of the $p_k$ can become zero for $\beta \to 1$. 
Let $W$ be the support of $p$, then the Lagrange equations of the problem, where $p$ is confined to $W$ 
reads 
$$ 2 g p = \lambda 1_W, (p,1_W) = 1 \; . $$
We don't yet know which subsets $W$ are selected by the thermodynamic functional $F$.

\paragraph{}
In other words, for selected subsets $W$ of the vertex set $V(G')=V(G_1)$
we get a critical point in the limit $\beta \to 1$. 
The zero temperature limit $\beta \to 1$ can therefore become complicated but 
the ``freezing process" appears to select some geometries. 
On the other hand, the high temperature regime is simple. What we don't yet know is 
whether there exists a universal constant $\beta_0 \in (0,1)$ such that for $\beta < \beta_0$
the free energy is analytic for any simplicial complex $G$. 
Classical statistical mechanics stories expect something
like this to happen: at high temperatures, entropy wins over energy and smooths out
free energy preventing bifurcations. It is also possible that $F$ needs to be rescaled depending on the
size of $G$ as in the infinite temperature case $\beta=0$, the measure $p$ is a uniform 
measure of weight $1/n$. Taking $p=1$ for example gives at $\beta=0$ always the 
free energy $\chi(G)$, independent of the Barycentric refinement. 

\begin{figure}[!htpb]
\scalebox{0.14}{\includegraphics{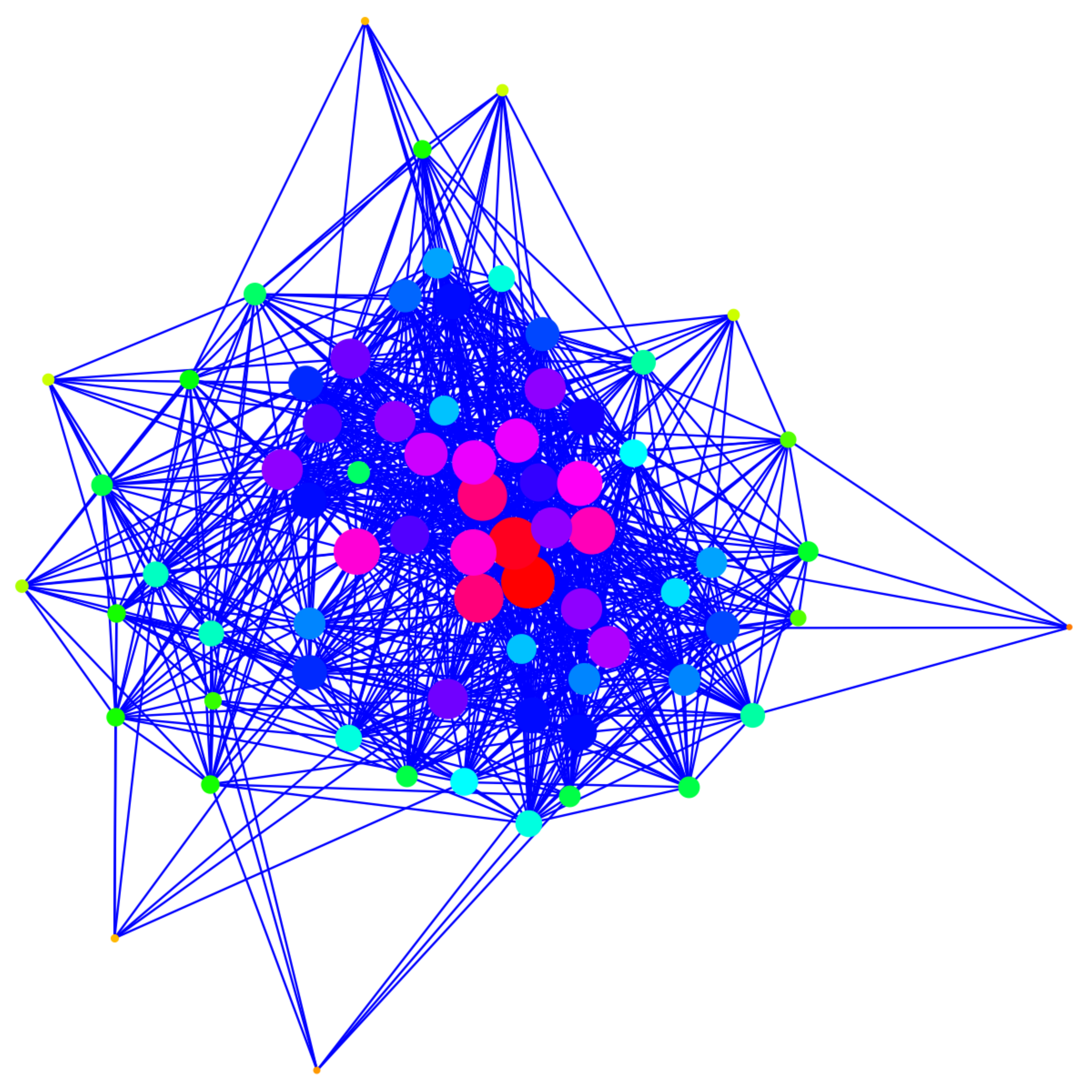}}
\scalebox{0.14}{\includegraphics{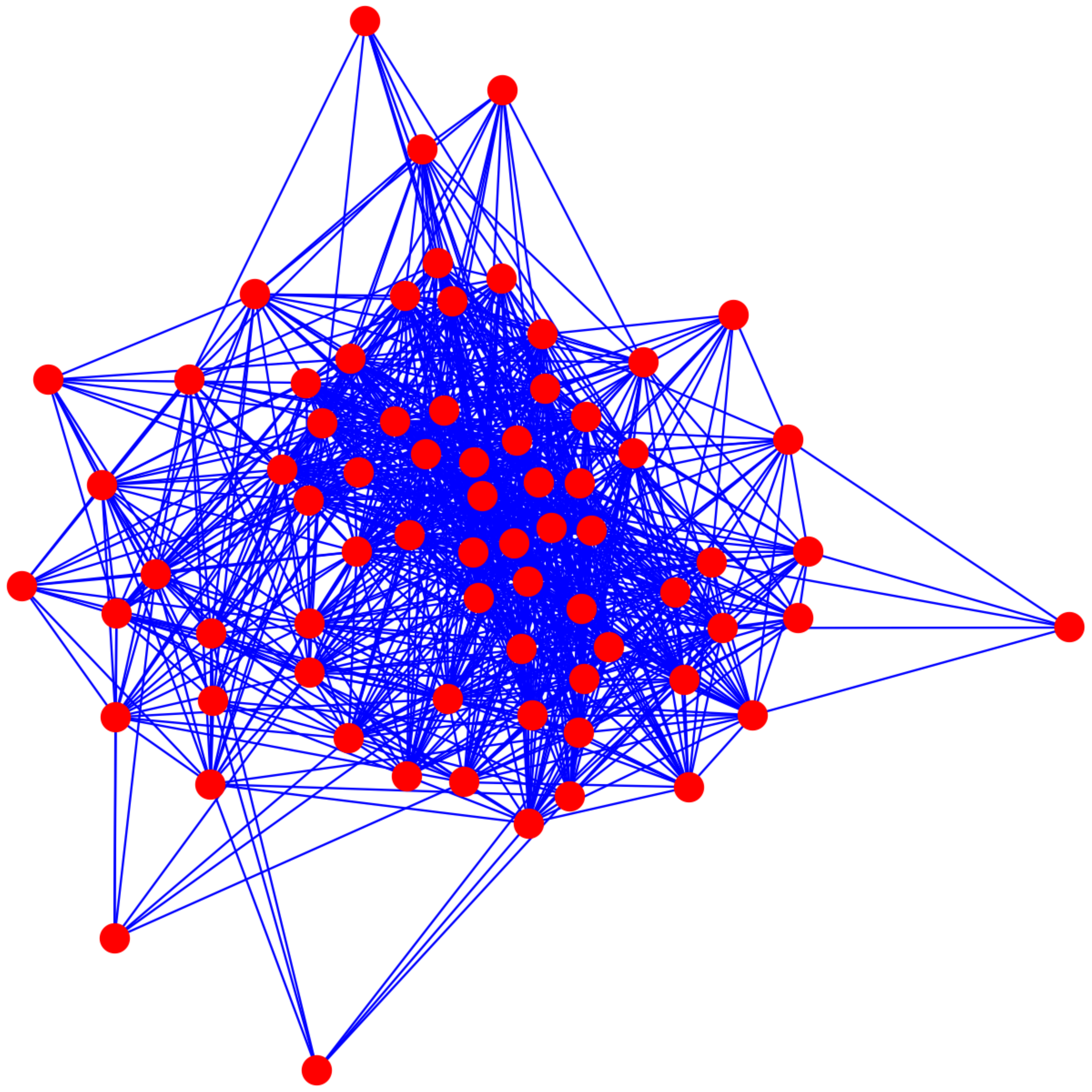}}
\caption{
We see the minimal free energy for the zero temperature case $\beta=1$, where $U$ is minimized. 
The complex $G$ is the Whitney complex of a random graph. 
Then to the right, we see the case of pure entropy $\beta=0$, the infinite temperature limit.
In the figure, the size of a vertex $x \in G'$ is drawn in size proportional to $p(x)$. 
}
\end{figure}

\section{Examples}

\paragraph{} If $G=K_2$, then 
$$ A' = \left[ \begin{array}{ccc} 1 & 1 & 1 \\ 1 & 1 & 0 \\ 1 & 0 & 1 \\ \end{array} \right]          $$
$$ g  = \left[ \begin{array}{ccc} -1 & 1 & 1 \\ 1 & 0 & -1 \\ 1 & -1 & 0 \\ \end{array} \right]  \; . $$
The minimal energy is obtained with $p=(3,2,2)/7$
The minimal entropy is $p=(1,1,1)/3$.
The Lagrange equations are
\begin{eqnarray*}
-T (\log(x_1)+1)-2 x_1+2 x_2+2 x_3 &=& \lambda \\
-T (\log(x_2)+1)+2 x_1-2 x_3 &=& \lambda \\
-T (\log(x_3)+1)+2 x_1-2 x_2 &=& \lambda \\
  x_1+x_2+x_3 &=& 1 \; .
\end{eqnarray*}
These are transcendental equations which we have to solve numerically. 

\paragraph{} If $G=K_3$, then 
$$ A' =  \left[ \begin{array}{ccccccc}
                   1 & 1 & 1 & 1 & 1 & 1 & 1 \\
                   1 & 1 & 0 & 0 & 1 & 1 & 0 \\
                   1 & 0 & 1 & 0 & 1 & 0 & 1 \\
                   1 & 0 & 0 & 1 & 0 & 1 & 1 \\
                   1 & 1 & 1 & 0 & 1 & 1 & 1 \\
                   1 & 1 & 0 & 1 & 1 & 1 & 1 \\
                   1 & 0 & 1 & 1 & 1 & 1 & 1 \\
                  \end{array} \right] \;  $$
$$ g =  \left[
                  \begin{array}{ccccccc}
                   1 & 1 & 1 & 1 & -1 & -1 & -1 \\
                   1 & 0 & 0 & 0 & 0 & 0 & -1 \\
                   1 & 0 & 0 & 0 & 0 & -1 & 0 \\
                   1 & 0 & 0 & 0 & -1 & 0 & 0 \\
                   -1 & 0 & 0 & -1 & 0 & 1 & 1 \\
                   -1 & 0 & -1 & 0 & 1 & 0 & 1 \\
                   -1 & -1 & 0 & 0 & 1 & 1 & 0 \\
                  \end{array} \right] \; . $$
The minimal energy for a measure with full support is obtained for 
$p=(7, 4, 4, 4, 6, 6, 6)/37$.
The minimal entropy is obtained with $p=(1,1,1,1,1,1,1)/7$.
The Lagrange equations are 
\begin{eqnarray*}
-T (\log(x_1)+1)+2x_1+2x_2+2 x_3+2 x_4-2 x_5-2 x_6-2 x_7&=&\lambda \\
-T (\log(x_2)+1)+2x_1-2x_7&=&\lambda \\
-T (\log(x_3)+1)+2x_1-2x_6&=&\lambda \\
-T (\log(x_4)+1)+2x_1-2x_5&=&\lambda \\
-T (\log(x_5)+1)-2x_1-2x_4+2 x_6+2x_7&=&\lambda \\
-T (\log(x_6)+1)-2x_1-2_3+2 x_5+2 x_7&=&\lambda \\
-T (\log(x_7)+1)-2x_1-2x_2+2 x_5+2 x_6&=&\lambda \\
x_1+x_2+x_3+x_4+x_5+x_6+x_7&=&1
\end{eqnarray*}
Also here, we have transcendental equations. We wrote down the equations
in full detail for $G=K_3$ to illustrate that also computer algebra systems can
not break the transcendental nature of the solutions. It would be great to
have explicit formulas for the possible solutions $p(\beta)$ minimizing the
Helmholtz free energy. As the problem is a finite dimensional variational problem,
rigorous bounds could be established using interval arithmetic. 

\paragraph{}
In the following picture, we look at the free energy $\beta U - T S$, where $T=1-\beta$.
When parameterized like this, we can go from $\beta=0$ to $\beta=1$. 

\begin{figure}[!htpb]
\scalebox{0.14}{\includegraphics{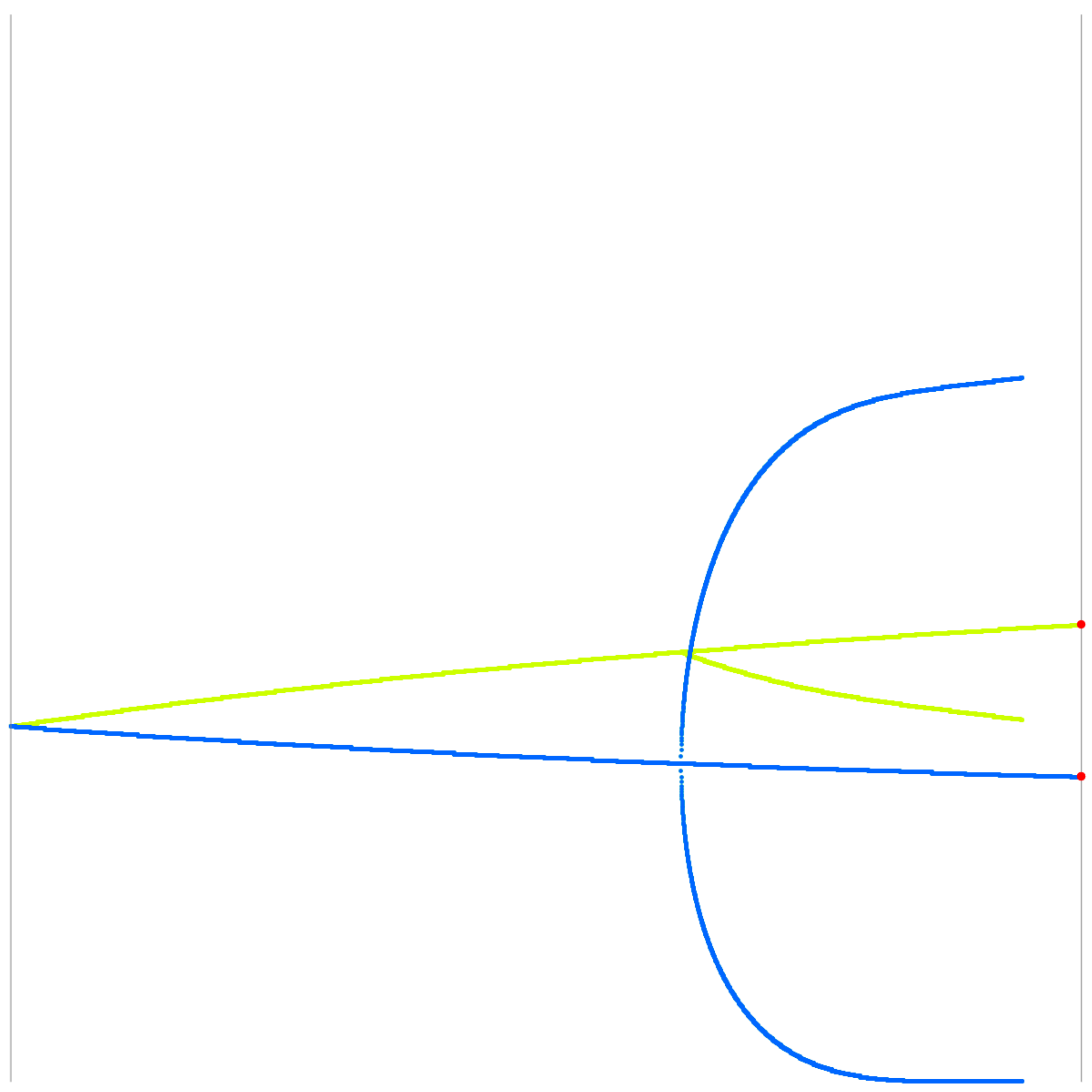}}
\scalebox{0.14}{\includegraphics{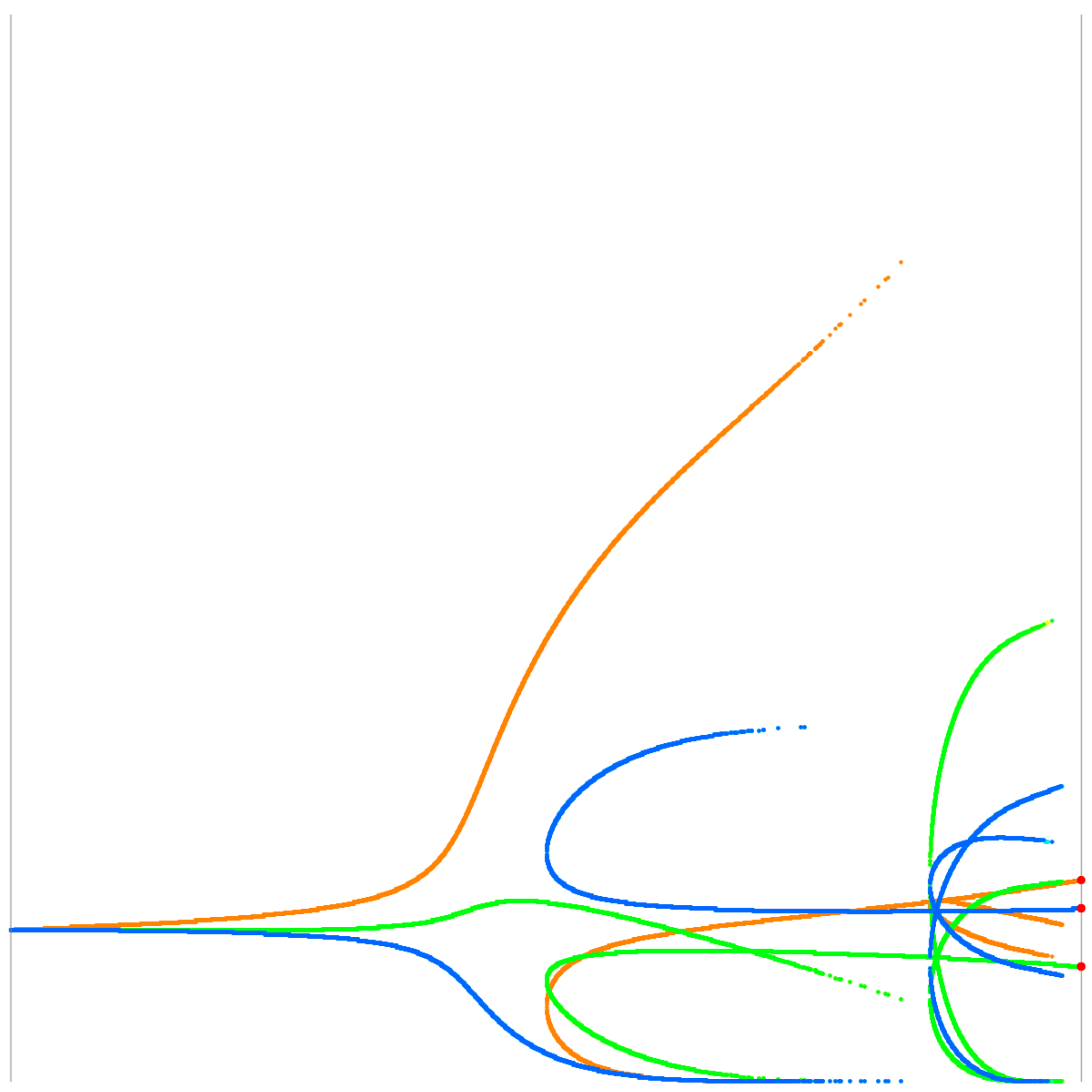}}
\caption{
The critical points $p(\beta)$ of the free energy $F(p) = \beta U(p) - (1-\beta) S(p)$ 
depending on inverse temperature $\beta$.
The left picture shows it for $G=K_2$ where one bifurcation value $\beta_1$ appears.
The right picture shows $\beta \to F(p(\beta),\beta)$ for $G=K_3$. We see two bifurcation
values $\beta_1,\beta_2$. The first is a {\bf saddle node bifurcation}, where two critical points
appear out of nothing, the second a {\bf pitchfork bifurcation}, where a critical point spans to
new critical points. For $\beta=0$, this is the pure entropy (high temperature) case, 
for $\beta=1$, it is the pure energy (zero temperature) case. Not all limiting measures
for $\beta \to 1$ have full support. 
}
\end{figure}

\section{Pushing curvature to $G$} 

\paragraph{}
For a finite simple graph $G$, let $V_i(x)$ denote the number of $K_{i+1}$
subgraphs in the unit sphere $S(x)$ of a vertex $x \in V$ with the understanding
that $V_{-1}(x)=1$. With the {\bf Euler curvature}
$$ K(v) = \sum_{k=0}^{\infty} \frac{(-1)^k V_{k-1}}{(k+1)} 
        = 1 - \frac{V_0}{2} + \frac{V_1}{3} - \frac{V_2}{4} + \cdots  $$
we have the Gauss-Bonnet formula $\sum_v K(v) = \chi(G)$ \cite{cherngaussbonnet}. 
The Euler characteristic $\chi(G)$ is
$$   \sum_{x \in V(G_1)} \omega(x)  \;  $$
so that $\omega(x) = (-1)^{{\rm dim}(x)}$ can be seen as a curvature too. If this
{\bf simplex curvature} $\omega(x)$ on a simplex is distributed equally to the
$k+1$ vertices of $x$, we get at each vertex the value $K(x)$. This is already the proof of
Gauss-Bonnet. It appeared in \cite{Levitt1992} but was not labeled as a Gauss-Bonnet result.

\paragraph{}
As we have seen that if $G$ was a finite simple curvature,  
the {\bf Euler curvature} was obtained by pushing the stable curvature $\omega(x)$
from the simplices $x$ to the vertices $v$. We can do the same than the curvature $K^+(x)$ 
and get the {\bf unstable Euler curvature}
$$  \tilde{K}(v) = \sum_{x, v \in x} (-1)^{\dim(x)} (1-\chi(S(x)))/|x|  \;. $$
As $\tilde{K}$ is a second order difference operator and only involves Euler characteristic again,
this looks like an interesting curvature, but

\begin{coro}
The unstable Euler curvature is the same than Euler curvature
$\tilde{K}(v) = K(v)$.
\end{coro} 
\begin{proof}
We see that by index averaging with a measure which is invariant
under $f \to -f$. The Euler curvature is an average over all functions $f$,
the unstable Euler curvature is an average over all functions $-f$.  
\end{proof}

\paragraph{}
Given a function $f$, we can push the curvature $K^-(x)$ from the simplices $x$ to the 
vertex $v$ in $x$ for which $f$ is maximal. 
This gives the Poincar\'e-Hopf index $1-S^-(x)$ \cite{poincarehopf}. 
This can also be done for $K^+(x)$, where it pushes the Poincar\'-Hopf index for $f$ on $G_1$ with a function $g$
on $G$ to the vertices of $G$. This gives then an curvature on the vertices of $G$. 

\section{Open problems}

{\bf A)} We still don't have a topological description of the entries $g(x,y)$ for $x \neq y$. 
We know that $g(x,y)$ is only non-zero if the unstable connection $W(x,y$ is not empty. 
We see this from the fact that $v=g e_x$ is the $x$ column of $g$. Then $L' v = e_x$. \\

{\bf B)} The bifurcation scheme of the free energy variational problem given by the 
Helmholtz free energy $F(p)= \beta U(p) - T S(p)$ appears to be interesting even in 
very concrete examples. The discontinuity of the function $\beta \to F(p_{min},\beta)$
can for every simplicial complex be determined case by case but it would be nice to know
where the bifurcation parameters $\beta_k(G)$ are, especially in the limit of Barycentric refinement.
The first bifurcation $\beta_1$ appears to grow with larger dimension.
Is there are Feigenbaum type universality \cite{Feigenbaum1978}? 

{\bf C)} We certainly have to explore more the nature of the Helmholtz free energy function 
$$  F(p,\beta) = \beta \sum_{x,y} g(x,y) p(x) p(y) + (1-\beta) \sum_x p(x) \log(p(x)  $$
for a simplicial complex $G$ with Green function $g=L'^{-1}=(1+A')^{-1}$
We see at critical points $p$ that $F_{\beta}(p(\beta),\beta) \leq 0$ and
$F(p(\beta),\beta)>0$ but more experiments are necessary. Having seen 
catastrophes \cite{StewartCatastrophe} already for $G=K_3$,
bifurcations, where the lowest free energy jumps discontinuously to a lower level. 
If $G$ has $n$ simplices meaning that $g$ is a $n \times n$ matrix and $p$ 
a stochastic vector in $\RR^n$,
then $F(p,0)=\log(n)$ and $F(p,1)=2 |E(G')|+|V(G')|=1/\lambda$, 
where $\lambda$ is the maximal eigenvalue of the {\bf San Diego Laplacian}
$\tilde{L}$ rsp the ground state energy of the San Diego Green function. So far we
have always seen that at all $\beta$ away from bifurcation values
and for all branches of the bifurcation, the function 
$\beta \to \partial_{\beta} F(p,\beta)$ is negative. \\

{\bf D)} We expect universal phenomena for suitably rescaled measures $p_{\infty}$ 
in the Barycentric limit $G_{\infty}$. 
A model close to a Barycentric refinement limit is the Hierarchical model by Dyson. \\

{\bf E)} A long shot is the hope that the Helmholtz free energy functional on simplicial 
probability complexes selects out interesting geometries $G$ and probability distributions
$\psi$ near temperatures selected out naturally. One can experiment then with physics of the
``gravity waves" $\psi(t) = e^{i L't} \psi$ similarly to the wave equations for 
the Hodge Laplacian $L=D^2$. For the later the solutions are given by Helmholtz 
$e^{i D} \psi= \cos(Dt) \psi(0) + i \sin(Dt) D^{-1} \psi'(0)$
solving the wave equation $\psi''=-L \psi$ responsible for non-gravitational parts. \\

{\bf F)} We still don't know whether the Fredholm operator $L'=(1+A')$ can be 
characterized somehow as the only $L$ which universally for all simplicial complexes
has a bounded inverse and which has the property that $L_{xy}=0$ if $x,y$ are 
disjoint. One could imagine for example to have $L_{xy}$ depend on the dimensions of
the simplices $x$ and $y$. \\

\begin{figure}[!htpb]
\scalebox{0.13}{\includegraphics{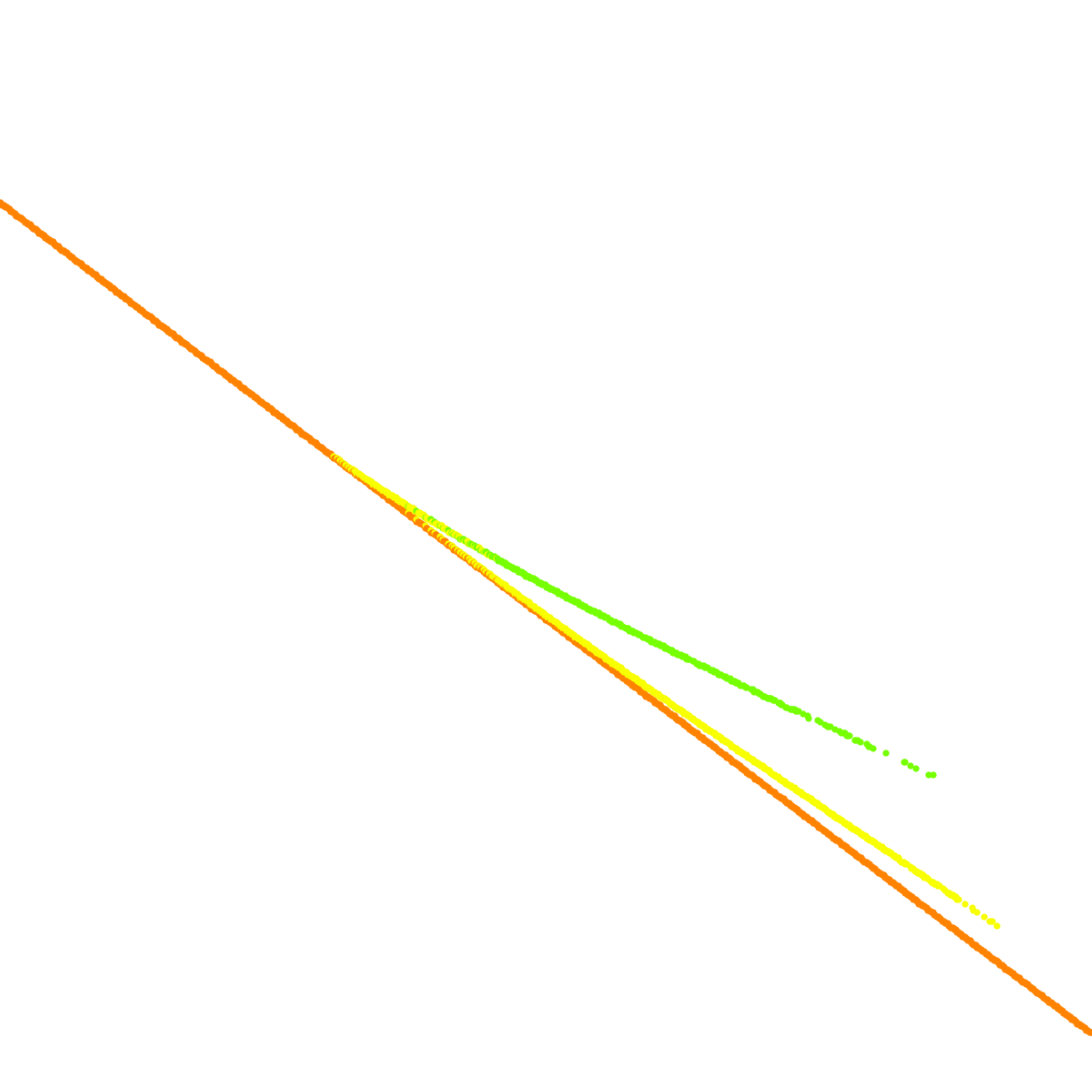}}
\scalebox{0.13}{\includegraphics{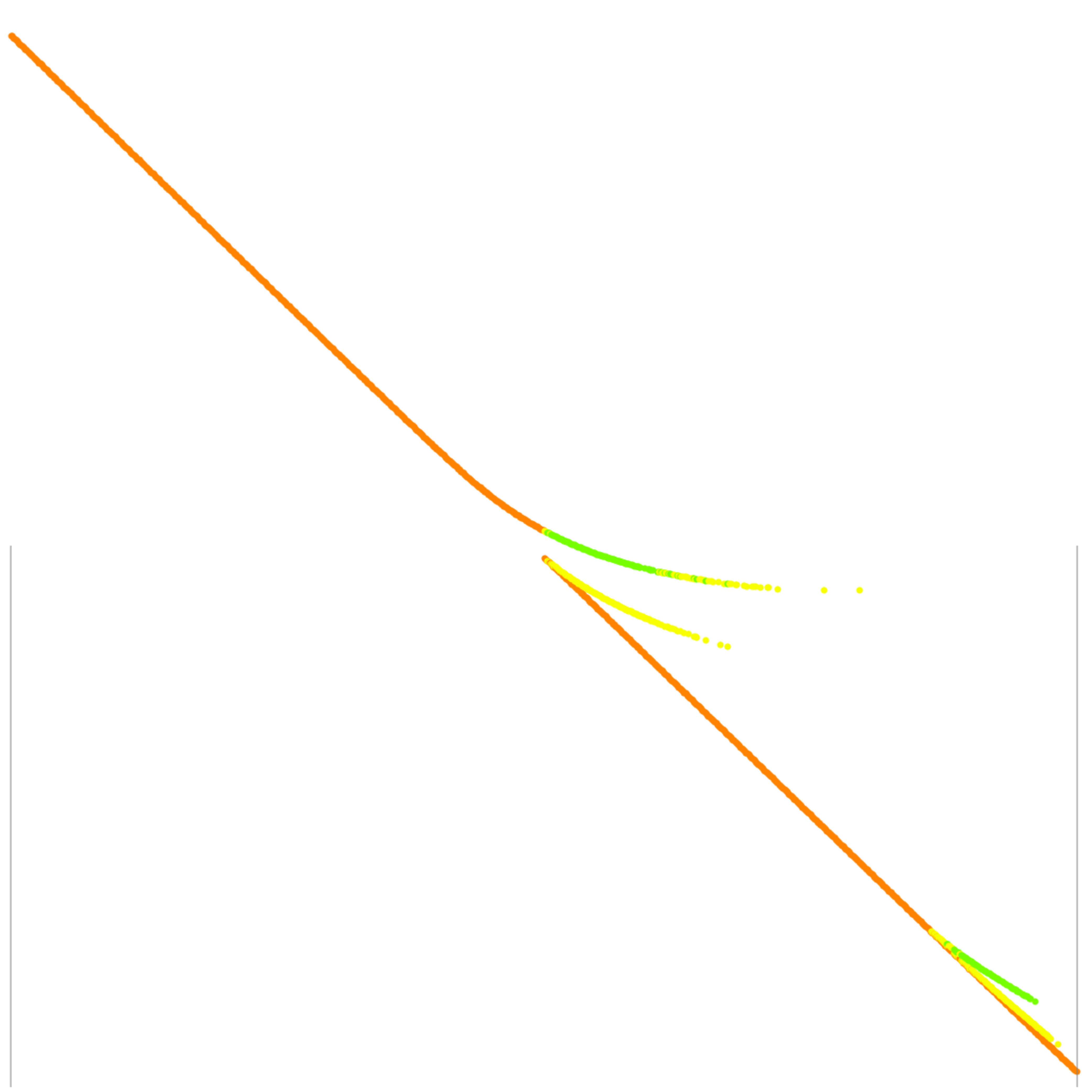}}
\caption{
We see $F(p,\beta)$ as a multi-valued function of $\beta \in [0,1]$ for $K_2$ and $G=K_3$. 
To the left, for $K_2$, we know $F(p,0)=\log(3)=1.099\dots$ and a critical point $p$ with $F(p,1)=1/7$ as
$2|E(G')|+|V(G')|=7$. 
For $K_3$ we know $F(p,0)=\log(7)=1.94\dots$, and a critical point $p$ with $F(p,1) = 1/37$ as
$2|E(G')|+|V(G')|=37$. More critical points are at the zero temperature limit $\beta=1$, 
as $p$ can be supported on a subgraph of $G'$. 
}
\end{figure}

\begin{figure}[!htpb]
\scalebox{0.13}{\includegraphics{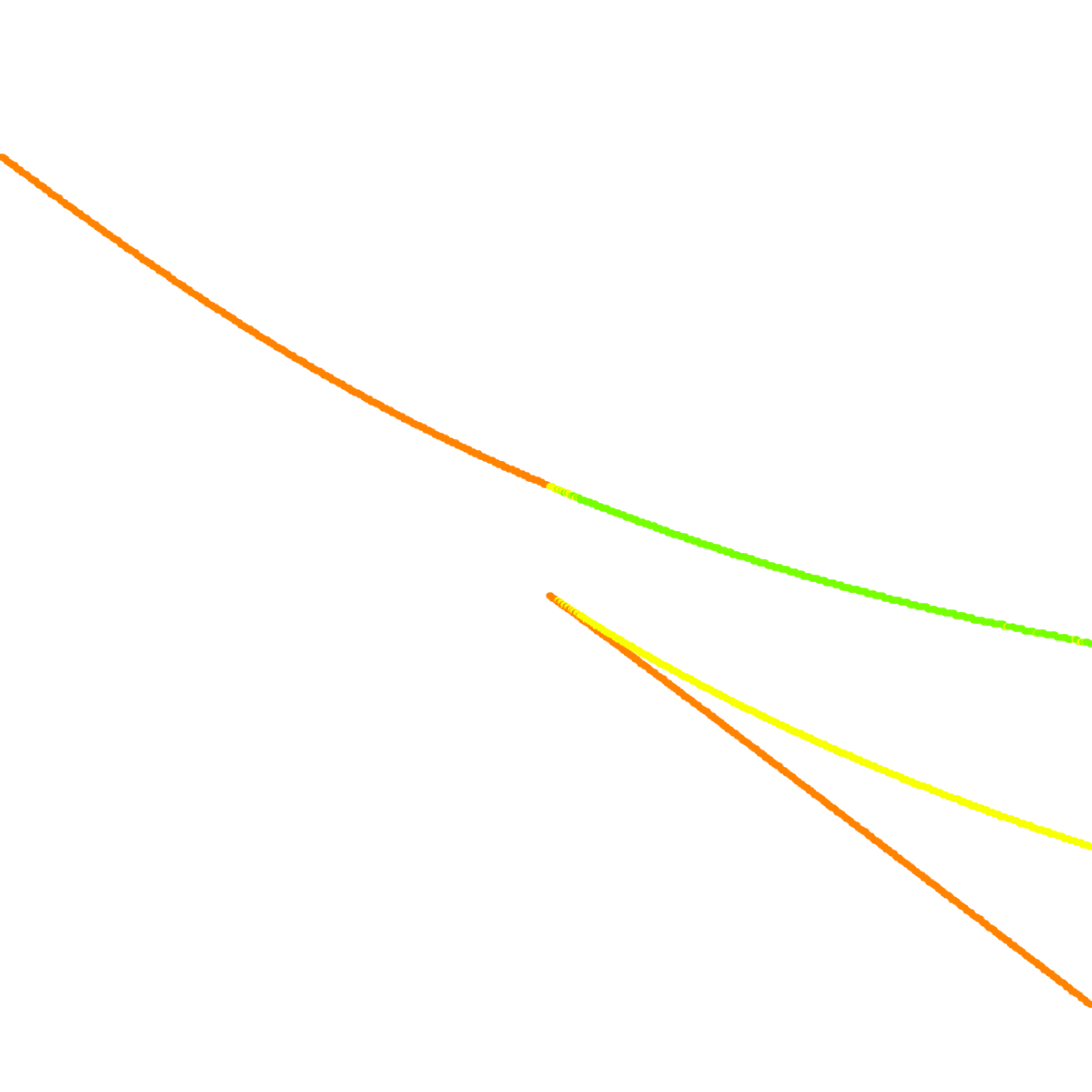}}
\scalebox{0.13}{\includegraphics{figures/freeenergy2.pdf}}
\caption{
We see the free energy $F$ for $G=K_3$ near the {\bf catastrophe} parameters. 
Again, $F$ is multi-valued due to the presence of different critical points in
some intervals. Catastrophe values are defined as 
parameter values $\beta_k$ at which the number of equilibrium points for $\beta-\epsilon$ and 
$\beta+\epsilon$ are different if
$\epsilon>0$ is small enough. There are three local extrema of the free energy functional $F(p,\beta)$ then. 
There is a second bifurcation, which is a pitchfork bifurcation for the measures. 
}
\end{figure}

\section{Appendix: on Euler characteristic and Shannon Entropy}

\paragraph{}
There would be many ways to modify or extend the Helmholtz functional $p \to F(p,\beta)$.
One could replace the potential energy $U$ with {\bf enthalpy} $U+P V$, where $P$ is an additional
{\bf pressure variable} and $V(G)$ is the {\bf volume}, the number of facets in $G$. This leads
to the {\bf Gibbs free energy}. An other possibility 
is to let the probability space be a space of fields $\psi$ rather than probabilities $p=|\psi|^2$ 
and using a {\bf Heisenberg interaction energy}
$\sum_{x,y} A'_{xy} \psi(x) \cdot \psi(y)$. We focus on an energy functional because 
$U(1) = \sum_{x,y} g(x,y) = \chi(G)$ is the Euler characteristic, which enjoys a nice uniqueness 
characterization. A similar uniqueness characterization applies also to the Shannon entropy 
functional $S$. We will outline this characterization here in this appendix. 

\paragraph{}
Before we point out the similarities between Euler characteristic and entropy, 
one can ask why it is natural to let energy and entropy compete in the form of the functional $F$. 
Helmholtz considered it a useful notion as it is relevant to various processes, especially in
chemistry, physics and cosmology; but Planck realized that the notion of entropy
has no accurate definition in cosmological terms. The usual informal definition as
a measure of the ``number of microscopic configurations" but that needs a finite 
probability space or a finite partition of the entire probability space and a conditional probability.
For absolutely continuous probability measures, the notion of {\bf differential entropy} makes
sense. A quantum mechanical version, the {\bf von Neumann entropy} 
deals with {\bf density operators} $P$, self adjoint operators
for which the eigenvalues $\lambda_j>0, j \in \NN$
add up to $1$. Then $S=-{\rm tr}(P \log(P))$ is the Shannon entropy 
for the probability measure on $\NN$ given by $p_i = \lambda_i$.

\paragraph{}
An other point of view came with Boltzmann who replaced the {\bf microcanonical ensemble}, the 
invariant measure on an energy surface of a Hamiltonian system with the
{\bf canonical ensemble} in which energy is no longer fixed but a new temperature variable like $\beta$ 
is introduced. On a calculus of variation level, it means replacing the energy functional 
with a Helmholtz free energy. While mathematical texts define it as such 
\cite{RuelleStatMech}, physics motivates
it by imagining the physical system placed into a ``heat bath" which means coupling it with a stochastic
system. But replacing the exact velocities of the particles with a statistical distribution
which depends on an inverse temperature parameter $\beta$ is a rather large step, more so than mean
field approaches. It is an emergent definition, similarly as Navier Stokes is an emergent PDE model from 
a $n$ body problem. Justifying the step would require to establish hyperbolicity for the mechanical system
which is not possible as for smooth interaction potentials, KAM theory has destroyed any hope as
there are often tiny parts of the phase space on which the dynamics remains integrable. Only in 
few cases, one has been able to prove ergodicity or even establish positive Kolmogorov-Sinai entropy. 

\paragraph{}
A finite abstract simplicial complex $G$ is a collection of non-empty sets closed under the process of 
taking non-empty subsets. A finite probability space is a finite set equipped with the algebra of
all subsets as events and a probability measure $p$. In the following, we just say simplicial complex 
or probability space when meaning finite abstract simplicial complex or finite probability space. 
What is a natural probability distribution on $G$? The most obvious one is to take the lowest energy 
state $\psi$ of the Laplacian  and take $p=|\psi|^2$. For the scalar Kirchhoff Laplacian $L_0=D-A$,
the constant distribution minimizes the energy. For the Fredholm connection matrix $1+A'$, we get more
interesting probability distributions and we have chosen to minimize $U(p) =\sum_{x,y} g(x,y) p(x) p(y)$
and we try here to justify the choice as $\sum_{x,y} g(x,y) = \chi(G)$ is a natural functional. 
Combining it with an other natural functional $S(p) = -\sum_x p(x) \log(p(x))$ is then natural too. 

\paragraph{}
In order to parallel Euler characteristic $\chi$ and entropy $S$ we restrict the functionals. 
On simplicial complexes, lets look at {\bf valuations}
on one side, functionals $\phi(G)$ of the form $\phi(G) = X \cdot f(G) = \sum_x X_{{\rm dim}(x)}$, where 
$f(G)$ is the $f$-vector of $G$ and $X=(X_0, \dots, X_{{\rm dim}(G)})$.
For a probability space $p$, we look at functionals $\phi(p) = \sum_x p(x) g(p(x))$, where $g$ is some function. 
In the following, if we say ``functional", we always mean a functional of this type in both cases.

\paragraph{}
The {\bf Cartesian product} of two finite abstract simplicial complexes $G$ and $H$ 
is the order complex of the Cartesian product $G \times H$. The {\bf Cartesian product} of two 
finite probability spaces $p$ on a finite set $G$ and $q$ on a finite set $H$ is the measure on 
the Cartesian product $G \times H$ of sets where the measure is defined by 
$p \times q( (x,y) ) = p(x) q(y)$. 

\paragraph{}
We say that a functional on simplicial complexes is multiplicative if $\phi(G \times H) = \phi(G) \phi(H)$.
We say that a functional on probability spaces is multiplicative if $\phi(p \times q) = \phi(p) \phi(q)$. 
Euler characteristic is an example of a multiplicative functional on simplicial complexes. 
The exponential of Shannon entropy is an example of a multiplicative functional on probability spaces. 

\paragraph{}
We say a functional on simplicial complexes is normalized if $\phi(K_1)=1$. 
We say a functional on probability measures is normalized if $\phi(p)=1$ if $p$ is a
probability measure supported on a single point. Euler characteristic is an example of a normalized
functional on simplicial complexes. The exponential of Shannon entropy is an example of a normalized
functional on probability spaces. 

\paragraph{} Euler characteristic is very natural, at least when restricting to 
valuations. We have seen that Wu characteristic is natural too when allowing
multi-linear valuations \cite{valuation}. 

\begin{propo}
Any multiplicative normalized functional on simplicial complexes is 
Euler characteristic.
\end{propo}
\begin{proof}
The multiplicative property implies especially that $\phi$ is invariant
under Barycentric refinements as $G_1=G \times K_1$ \cite{KnillProduct}. 
There is an explicit 
Barycentric refinement operator $A$ which maps $f_G$ to $f_{G_1}$. 
This operator has only one eigenvector $X=(1,-1,1, \dots)$  \cite{valuation}
implying that $\phi$ must be the Euler characteristic. 
\end{proof}

\paragraph{} Also the Shannon entropy functional is natural. There is a tiny
ambiguity about the choice of the base of the logarithm but this can be bootstrapped
once we know that entropy is a natural functional:

\begin{propo}[Shannon]
Any multiplicative normalized functional on probability spaces must
be some exponential of entropy. 
\end{propo}
\begin{proof}
This is essentially theorem 2 in Shannon \cite{Shannon48}.
Lets look at a product of two probability spaces, with measures $(p,q)$ and $(a,b)$. 
The requirement 
$pa g(pa)+pb g(pb) + qa g(qa)+qb g(qb) = (p g(p) + q g(q)) (a g(a) + b g(b))$
implies $g(pa) = g(p) g(a)$ so that $g(x) = \log_b(x)$ for some base $b$. 
\end{proof} 

\paragraph{}
There is still the question about the choice of the base $b$ of the logarithm or the 
choice of the exponential. Now which real number minimizes $-x \log_b(x)$, where $\log_b$ 
is the logarithm to any arbitrary base $b$. The answer to this extremal problem
is $1/e$, independent of $b$. As maximal entropy selects out the base,
lets take this as a base for entropy. We get $g(x)=\log_{1/e}(x) = - \log(x)$, where
$\log$ is of course the natural logarithm. 

\paragraph{}
Having singled out Euler characteristic which is $\sum_{x,y} g(x,y)$, it is natural
to take the energy $U(p) = \sum_{x,y} g(x,y) p(x) p(y)$. 
One could ask why not take $U(p) - e^{T S(p)}$ and suspect that it should not matter much
like in Maupertius principles, where extremizing the length or energy functional leads to
equivalent critical points. We see indeed that the situation remains essentially unchanged
for very small temperatures $T$ but that at high temperature
the bifurcations of the critical points of the Helmholtz functional start much earlier. 
We stick to the standard $U-TS$ functional also because it has proven to be so fundamental
in other domains. 

\paragraph{}
To summarize, we have argued in this appendix that the Helmholtz functional 
$$ F(G,p) = \beta \sum_{x,y} g(x,y) p(x) p(y) + (1-\beta) \sum_x p(x) \log(p(x)) $$
is a natural functional on finite abstract simplicial complexes equipped with a 
probability measure $p$. Whether it is useful to describe some phenomena in 
nature or select interesting geometries by ``placing the complex $G$ into a heat bath,
and then turning the temperature to zero, picking the lowest energy state limit`` 
still needs to be explored. Encouraging are the two main results of this note,
the energy-topology connection $\sum_{x,y} g(x,y) = \chi(G)$ as well as that
the path from $\beta=0$ to $\beta=1$ features catastrophes already for small 
complexes $G$.

\bibliographystyle{plain}

\begin{thebibliography}{10}

\bibitem{BowenLanford}
R.~Bowen and O.~E. {L}anford, III.
\newblock Zeta functions of restrictions of the shift transformation.
\newblock In {\em Global {A}nalysis ({P}roc. {S}ympos. {P}ure {M}ath., {V}ol.
  {XIV}, {B}erkeley, {C}alif., 1968)}, pages 43--49. Amer. Math. Soc., 1970.

\bibitem{Chung97}
F.~Chung.
\newblock {\em Spectral graph theory}, volume~92 of {\em CBMS Regional Conf.
  Series}.
\newblock AMS, 1997.

\bibitem{Cycon}
H.L. Cycon, R.G.Froese, W.Kirsch, and B.Simon.
\newblock {\em {Schr\"odinger} Operators---with Application to Quantum
  Mechanics and Global Geometry}.
\newblock Springer-Verlag, 1987.

\bibitem{Feigenbaum1978}
M.J. Feigenbaum.
\newblock Quantitative universality for a class of nonlinear transformations.
\newblock {\em J. Statist. Phys.}, 19(1):25--52, 1978.

\bibitem{cherngaussbonnet}
O.~Knill.
\newblock A graph theoretical {Gauss-Bonnet-Chern} theorem.
\newblock {\\}http://arxiv.org/abs/1111.5395, 2011.

\bibitem{poincarehopf}
O.~Knill.
\newblock A graph theoretical {Poincar\'e-Hopf} theorem.
\newblock {\\} http://arxiv.org/abs/1201.1162, 2012.

\bibitem{indexformula}
O.~Knill.
\newblock An index formula for simple graphs \hfill.
\newblock {\\}http://arxiv.org/abs/1205.0306, 2012.

\bibitem{indexexpectation}
O.~Knill.
\newblock On index expectation and curvature for networks.
\newblock {\\}http://arxiv.org/abs/1202.4514, 2012.

\bibitem{knillmckeansinger}
O.~Knill.
\newblock {The McKean-Singer Formula in Graph Theory}.
\newblock {\\}http://arxiv.org/abs/1301.1408, 2012.

\bibitem{KnillFunctional}
O.~Knill.
\newblock Characteristic length and clustering.
\newblock {{\\}http://arxiv.org/abs/1410.3173}, 2014.

\bibitem{knillgraphcoloring}
O.~Knill.
\newblock Coloring graphs using topology.
\newblock {{\\}http://arxiv.org/abs/1410.3173}, 2014.

\bibitem{KnillBarycentric}
O.~Knill.
\newblock The graph spectrum of barycentric refinements.
\newblock {{\\}http://arxiv.org/abs/1508.02027}, 2015.

\bibitem{KnillProduct}
O.~Knill.
\newblock The {K}uenneth formula for graphs.
\newblock {{\\}http://arxiv.org/abs/1505.07518}, 2015.

\bibitem{KnillSard}
O.~Knill.
\newblock A {S}ard theorem for graph theory.
\newblock {{\\}http://arxiv.org/abs/1508.05657}, 2015.

\bibitem{KnillBarycentric2}
O.~Knill.
\newblock Universality for barycentric subdivision.
\newblock {{\\}http://arxiv.org/abs/1509.06092}, 2015.

\bibitem{valuation}
O.~Knill.
\newblock Gauss-{B}onnet for multi-linear valuations.
\newblock {\\}http://arxiv.org/abs/1601.04533, 2016.

\bibitem{Unimodularity}
O.~Knill.
\newblock On {F}redholm determinants in topology.
\newblock {\\}https://arxiv.org/abs/1612.08229, 2016.

\bibitem{Spheregeometry}
O.~Knill.
\newblock Sphere geometry and invariants.
\newblock {\\}https://arxiv.org/abs/1702.03606, 2017.

\bibitem{Levitt1992}
N.~Levitt.
\newblock The {E}uler characteristic is the unique locally determined numerical
  homotopy invariant of finite complexes.
\newblock {\em Discrete Comput. Geom.}, 7:59--67, 1992.

\bibitem{McKeanSinger}
H.P. McKean and I.M. Singer.
\newblock Curvature and the eigenvalues of the {L}aplacian.
\newblock {\em J. Differential Geometry}, 1(1):43--69, 1967.

\bibitem{StewartCatastrophe}
T.~Poston and I.~Stewart.
\newblock {\em Catastrophe theory and its applications}.
\newblock Pitman, 1978.

\bibitem{RuelleStatMech}
D.~Ruelle.
\newblock {\em Statistical Mechanics}.
\newblock Mathematics Physics Monograph Series. W.A. Benjamin, Inc, 1968.

\bibitem{Shannon48}
C.E. Shannon.
\newblock A mathematical theory of communication.
\newblock {\em The Bell System Technical Journal}, 27:379--423,623--656, 1948.

\bibitem{SimonStatMechanics}
B.~Simon.
\newblock {\em The statistical mechanics of lattice gases}, volume Volume I.
\newblock Princeton University Press, 1993.

\end{thebibliography}

\end{document}